\documentclass[12pt,oneside]{amsart}
\usepackage[export]{adjustbox}
\usepackage{geometry}                % See geometry.pdf to learn the layout options. There are lots.
\usepackage{graphicx}
\usepackage{amssymb}
\usepackage{epstopdf}
\usepackage{float}
\usepackage{wrapfig}
\usepackage{comment}
\usepackage{bbm}
\usepackage{fourier}
\usepackage{amsmath,amsthm,amscd,amssymb}
\usepackage{latexsym}
\usepackage[colorlinks,citecolor=red,pagebackref,hypertexnames=false]{hyperref}
\numberwithin{equation}{section}
\theoremstyle{plain}
\newtheorem{theorem}{Theorem}[section]
\newtheorem{lemma}[theorem]{Lemma}

\newtheorem{proposition}[theorem]{Proposition}
\newtheorem{conjecture}[theorem]{Conjecture}

\theoremstyle{definition}
\newtheorem{definition}[theorem]{Definition}

\newtheorem{example}[theorem]{Example}

\theoremstyle{remark}
\newtheorem{fact}[theorem]{Fact}
\newtheorem{remark}[theorem]{Remark}

\newtheorem{case[theorem]}{Case}

\newcommand*{\ind}{\mathbbm{1}}

\usepackage{color}
\definecolor{blue}{rgb}{0,0,1}

\definecolor{red}{rgb}{1,0,.2}

\usepackage{soul,xcolor}

\date{{\today}}
\author{K\'aroly Simon}
\address[K\'aroly Simon]{Budapest University of Technology and Economics, Department of Stochastics, Institute of Mathematics and
MTA-BME Stochastics Research Group,
 1521 Budapest, P.O.Box 91, Hungary} \email{simonk@math.bme.hu}

\author{Lajos V\'ag\'o}
\address[Lajos V\'ag\'o]{Budapest University of Technology and Economics, Department of Stochastics,
 1521 Budapest, P.O.Box 91, Hungary} \email{vagolala@gmail.com}

 \thanks{
\indent
 The research of Simon was
 partially supported by the  grant 346300 for IM PAN from the Simons Foundation and the matching 2015-2019 Polish MNiSW fund
 and also partially
  supported by OTKA Foundation
\# K 123782,
 and  MTA TKI Stochastics Reserch Group.
}

\title{\parbox{14cm}{\centering{ Fractal Percolations }}}

\begin{document}

%\setstcolor{red}

\begin{abstract}
One of the most well known random fractals is the so-called Fractal percolation set. This is defined as follows: we divide the unique cube in $\mathbb{R}^d$ into $M^d$ congruent sub-cubes. For each of these cubes a certain retention  probability is assigned with which we retain the cube. The interior of the discarded cubes contain no points from the Fractal percolation set.  In the retained ones we repeat the process ad infinitum. The set that remains after infinitely many steps is the Fractal percolation set.
The homogeneous case is when all of these probabilities are the  same. Recently, there have been considerable developments in regards with the projection and slicing properties in the homogeneous case. In the first part of this note we give an account of some of these recent results and then we discuss the difficulties and provide some new partial results in the non-homogeneous case.
\end{abstract}

\maketitle

\section{Introduction}

This is a survey of some classical, recent and even new  results of a family of random fractal sets introduced by
B. Mandelbrot \cite{Mandelbrot1974}, \cite{Mandelbrot1983}.
In the literature this family  of random Cantor sets  is called Fractal percolation, Mandelbrot percolation and sometimes canonical curdling. This survey is based on the lectures the first author gave on the Dynamical Systems, Simons Semester in the Autumn of 2015 at IMPAN.
This survey focuses on the most recent developments and complements two recent, more general surveys \cite{chayes1995aspects}, \cite{Dekking2009} and \cite{rams2014geometry}, and the corresponding chapters in the book \cite{bishop2016fractals}.

\section{The definition of the Fractal percolation Sets}\label{a74}
 The parameters of the construction are
 \begin{itemize}
   \item $d$ which gives  the dimension where we work,
   \item $M$: we partition the unit cube $Q:=[0,1]^d$ into $M^d$ congruent $1/M$ mesh-cubes
       $$
       K_{\mathbf{i}}:=\left[\frac{i_1}{M},\frac{i_1+1}{M}\right]
       \times \cdots \times
       \left[\frac{i_d}{M},\frac{i_d+1}{M}\right],
   $$
   where $\mathbf{i}\in \mathcal{I}:=\left\{0, \dots ,M-1\right\}^d$
   \item $\mathbf{p}:=\left\{p_{\mathbf{i}}\right\}_{\mathbf{i}
       \in\mathcal{I}}$, the vector of probabilities  $p_{\mathbf{i}}\in [0,1]$. ($\mathbf{p}$ is not a probability vector in general).
 \end{itemize}
 For every $\mathbf{i}\in \mathcal{I}$ let
 $S_{\mathbf{i}}(x)$ be the similarity transformation which sends
 $Q$ onto $K_\mathbf{i}$ in the most natural way.

 The idea of the construction is that in the first step of the construction, for every $\mathbf{i}\in\mathcal{I}$, we retain every $K_\mathbf{i}$ with probability $p_\mathbf{i}$ independently and in every retained cube $K_{\mathbf{i}}$ we repeat the same process independently of everything. What  remains after infinitely many steps it is the Fractal percolation random set, which may be the emptyset. Now we write down the construction in a formal way.
 \subsection{The construction of the fractal percolation
  Cantor set} \label{p99}
Consider the $M^d$-array tree $\mathcal{T}$. The root is $ \emptyset $ and the nodes of $\mathcal{T}$ are the  finite words above the alphabet $\mathcal{I}$. For a level $n$ node
$\mathfrak{i}=(\mathbf{i}_1, \dots\mathbf{i}_n)$, where $\mathbf{i}_k\in\mathcal{I}$, $k=1, \dots ,n$ we define the corresponding:
\begin{itemize}
  \item map $S_{\mathfrak{i}}:=S_{\mathbf{i}_1}\circ\cdots \circ S_{\mathbf{i}_1}$
  \item level $n$-cube $K_{\mathfrak{i}}:=S_{\mathfrak{i}}(Q)$
  \item random variable
$X_{\mathfrak{i}}$ such that
$$
\mathbb{P}\left(X_{\mathfrak{i}}=1\right):=
p_{\mathbf{i}_n},\mbox{ and }
\mathbb{P}\left(X_{\mathfrak{i}}=0\right):=
1-p_{\mathbf{i}_n},
\quad
\mathfrak{i}=(\mathbf{i}_1, \dots\mathbf{i}_n)
$$
and
$$
\left\{X_\mathfrak{i}\right\}_{i\in\mathcal{T}} \mbox{ are independent.}
$$
We define $X_{ \emptyset }:=1$.
\end{itemize}
The set of the  labels of retained level $n$ cubes is
$$
\mathcal{E}_n:=
\mathcal{E}_n(d,M,\mathbf{p}):=\left\{
\mathfrak{i}=(\mathbf{i}_1, \dots\mathbf{i}_n)\in \mathcal{I}^n:
X_{\mathbf{i}_1}\cdot X_{\mathbf{i}_1\mathbf{i_2}}
\cdots X_{\mathbf{i}_1, \dots\mathbf{i}_n}=1
\right\}.
$$
The $n$-th approximation $E_n(d,M,\mathbf{p})$ of the Fractal percolation set $E(d,M,\mathbf{p})$ is
$$
E_n(d,M,\mathbf{p}):=\bigcup\limits_{\mathfrak{i}
\in\mathcal{E}_n}K_{\mathfrak{i}}.
$$
 Finally the Fractal percolation set corresponding to parameters
 $(d,M,\mathbf{p})$ is defined by
 \begin{equation*}
  E:= E(d,M,\mathbf{p}):=\bigcap\limits_{n=1}^{\infty }E_n(d,M,\mathbf{p}).
 \end{equation*}

 \begin{definition}
If all the probabilities $p_\mathbf{i}\equiv p$, then we say that $E$ is homogeneous in probabilities, otherwise $E$ is inhomogeneous in probabilities. In the literature the term Fractal percolation is used mostly for the homogeneous in probability case.
 \end{definition}

 \begin{remark}
   Clearly, $(\#\mathcal{E}_n)$ is a branching process with mean $\mathbb{E}\left[\#\mathcal{E}_1\right]=
   \sum\limits_{\mathbf{i}\in\mathcal{I}}p_\mathbf{i}$. In the homogeneous case when $p_\mathbf{i}\equiv p$, the offspring distribution is $\mathrm{Binomial}(M^d, p)$.
   \end{remark}
Therefore in what follows we always assume that
\begin{equation}\label{a94}
 \sum\limits_{\mathbf{i}\in\mathcal{I}}p_{\mathbf{i}}>1,
\end{equation}
because otherwise the process dies out in finite steps almost surely.
The following Theorem was proved by Falconer, Mauldin and Williams:
\begin{theorem}\label{a96} Conditioned on non-extinction (that is $E\ne  \emptyset $)
\begin{equation}\label{a95}
   \dim_{\rm H} E=\dim_{\rm B}=\frac{\log \sum\limits_{\mathbf{i}\in\mathcal{I}}p_{\mathbf{i}}}{\log M} \text{holds almost surely.}
\end{equation}
\end{theorem}
This is why we say that the expected dimension of $E$ is $\frac{\log \sum\limits_{\mathbf{i}\in\mathcal{I}}p_{\mathbf{i}}}{\log M}$. We will frequently use the expression that the expected dimension of $E$ is greater than $1$. Clearly, this is equivalent to

\begin{equation}\label{a93}
     \sum\limits_{\mathbf{i}\in\mathcal{I}}p_{\mathbf{i}}>M.
\end{equation}
\begin{definition}[Natural measure in the homogeneous case]\label{a77}
Here we assume that all $p_{\mathbf{i}}\equiv p$. Then $\mathbb{E}\left[\# \mathcal{E}_n\right]=\left(M^d \cdot p\right)^n$. So we consider
$$
Z:=\lim\limits_{n\to\infty} \frac{\#\mathcal{E}_n}{\left(M^d \cdot p\right)^n}.
$$
As always we assume that $p>1/M^d$. It is well known from the theory of branching processes that
\begin{equation}\label{a76}
  \mathbb{P}\left(Z=0|\mathrm{non-extinction}\right)=0
  \text{ and }
  \mathbb{E}\left[Z\right]=1.
\end{equation}
It was proved by Mauldin and Williams in \cite{Mauldin1986} that 
so-called  natural measure $\mu$ defined below exists:
\begin{equation}\label{a75}
\mu:=  \lim\limits_{n\to\infty}
  \frac{\mathcal{L}_d|_{E_n}}{p^n \cdot Z}.
\end{equation}
\end{definition}

First we collect some of the interesting results related to the homogeneous case then we focus on the inhomogeneous one.

\section{Homogeneous probabilities}

Throughout this section we collect recent results on fractal percolation constructed with homogeneous probabilities, i.e., $p_\mathbf{i} = p$ for all $\mathbf{i}$. If the ambient dimension $d$ is considered fixed and $M$ is also fixed then we write $E(p)$ for the fractal percolation set $E(d,M,\mathbf{p})$ in this case.

\subsection{Projections of the fractal percolation set
and its natural measure}

Orthogonal projections of fractal percolation sets parallel to coordinate axes were investigated by Falconer and Grimmett \cite{Falconer1992, Falconer1994}. They obtained sufficient and necessary conditions for the existence of interior points of the projected set. It took roughly two decades to extend these results to other directions as well. Fist we need some notation:

\begin{definition}\label{u98}
 Let $\alpha:=\left\{\mathbf{a}^1, \dots ,\mathbf{a}^k\right\}$ be a k-tuple of orthonormal vectors in $\mathbb{R}^d$
for some $k\in\left\{1, \dots ,d-1\right\}$.
The collection of all such  $\alpha$ is denoted by $\mathcal{A}_{d,k}$.
We write $S_\alpha$ for the subspace of $\mathbb{R}^d$ spanned by
$\alpha$ and  $ \Pi_\alpha$ is the orthogonal projection from $\mathbb{R}^d$ to
$S_\alpha$. The assumption of the following theorem is  a condition which we call condition  $A(\alpha)$ (see Definition \ref{Balpha}). It was defined in \cite{Rams} in $\mathbb{R}^2$. The authors of this note
 extended the definition of condition $A(\alpha)$
  to higher dimension in the natural way (see
\cite{simon2014projections}).
\end{definition}

\begin{theorem}\cite{Rams, simon2014projections}\label{thm:interval}
Fix $d\geq 2$, $1\leq k<d$. If\  Condition $A(\alpha)$ holds for an $\alpha\in\mathcal{A}_{d,k}$ then $\Pi_\alpha(E)$ has nonempty interior, conditioned on $E\neq \emptyset$.
Hence,
if  Condition $A(\alpha)$ holds for all $\alpha\in \mathcal{A}_{d,k}$ then almost every realization of the fractal percolation set,
for \emph{all} $\alpha$ orthogonal projections $\Pi_\alpha(E)$ has nonempty interior, conditioned on $E\neq \emptyset$.
\end{theorem}

For example, if we are on $\mathbb{R}^2$, in the homogeneous case when all probabilities are equal to a $p>\frac{1}{M}$  a typical non-empty realisation of the fractal percolation contains an interval in  its angle $\alpha$ projection for all $\alpha$. Although the proof is quite technical, it basically consists of two main parts. The authors of \cite{Rams} consider the intersections of lines $\ell_\alpha(x)$ with level-$n$ approximations $E_n$. On the one hand, Condition $A(\alpha)$ guarantees that the expected number of level-$n$ small squares in $E_n$ intersecting $\ell_\alpha(x)$ increase exponentially with $n$. On the other hand, a large deviation argument shows that this expected behavior really realises almost surely, simultaneously for all directions.

The paper of Peres and Rams \cite{Peres2016} evolves the above argument to obtain information about the projection of the natural self-similar measure as well.

\begin{theorem}[Peres, Rams]\label{thm:Peres}\label{0334}
For planar fractal percolation (i.e., $d=2$) constructed with homogeneous probabilities, almost surely, conditioned on $E \neq \emptyset$, all orthogonal projections of the natural measure are absolutely continuous and except for the horizontal and vertical projection have H{\"o}lder continuous density.
\end{theorem}

\subsection{Dimension conservation}

In the last years several authors received dimension conservation results on projections of fractal percolation sets \cite{falconer2014exact, falconer2015dimcons, shmerkin_si}. First let us introduce three dimension conservation notions relevant for this paper. In all three definitions $A\subset \mathbb{R}^d$ and $\Pi_\alpha$ are orthogonal projections from $\mathbb{R}^d$ to $k$-dimensional planes parametrized by $\alpha\in \mathcal{A}_{d,k}$, $1\leq k < d$ as in Definition \ref{u98}. For simplicity we identify the image space $S_\alpha$ of each $\Pi_\alpha$ with $\mathbb{R}^k$. The following  definition is due to Furstenberg \cite{furstenberg2008}:

\begin{definition}\cite{furstenberg2008}
We say that $\Pi_\alpha$ is \emph{dimension conserving} for $A$ if there is $\varepsilon > 0$ for which
\begin{equation}\label{furst_dim_cons}
  \varepsilon + \dim_{\rm H}  \{ x \in S_\alpha\ :\ \dim_{\rm H} (A \cap \Pi_\alpha^{-1}x) \geq \varepsilon\} \geq \dim_{\rm H} A.
\end{equation}
\end{definition}
The second, slightly different notion is by Falconer and Jin \cite{falconer2015dimcons}.

\begin{definition}\cite{falconer2015dimcons}
We say that $\Pi_\alpha$ is \emph{weakly dimension conserving} for a set $A \subset \mathbb{R}^d$ if $\forall \varepsilon>0$:
\begin{equation}\label{weak_dim_cons}
 \dim_{\rm H}  (A \cap \Pi_\alpha^{-1}x) \geq \dim_{\rm H}  A - k - \varepsilon, \qquad
 \forall x\in W,
\end{equation}
where $W\subset S_\alpha$ has $\mathcal{L}_k(W) > 0$.
\end{definition}

Finally, strong dimension conservation is defined by Shmerkin and Suomala \cite{shmerkin_si}.
\begin{definition}\cite{shmerkin_si}
We say that $\Pi_\alpha$ is \emph{strongly dimension conserving} for a set $A \subset \mathbb{R}^d$ if there is a nonempty open set $U \subset S_\alpha$ such that
\begin{equation*}
  \dim_{\rm H}  (A \cap \Pi_\alpha^{-1}x) = \dim_{\rm H}  A - k
\end{equation*}
for all $x \in U$.
\end{definition}
Clearly, strong dimension conservation implies both \eqref{furst_dim_cons} with $\varepsilon =\dim_{\rm H}  A - k$ and \eqref{weak_dim_cons}.

Falconer and Jin \cite{falconer2014exact} investigated random multiplicative cascade measures defined on self-similar sets, which includes measures supported on the fractal percolation set. Their results imply dimension conservation in the classical Furstenberg sense \eqref{furst_dim_cons} as a corollary of exact dimensionality results \cite[Theorem 3.1]{falconer2014exact}. Importantly, they require no separation condition. In particular, for fractal percolation their results imply the following.
\begin{theorem}\cite[Corollary 1]{falconer2014exact}
For any fixed direction $\alpha \in \mathcal{A}_{d,k}$ a.s.
\[
  \dim_{\rm H}  \mu_\alpha + \dim_{\rm H}  \mu_{x, \alpha} = \dim_{\rm H}  \mu
\]
for $\mu_\alpha$ almost all $x \in \Pi_\alpha K$, where $\mu_\alpha$ is the push forward of $\mu$ by $\Pi_\alpha$ and $\mu_{x,\alpha}$ is the corresponding conditional measure.
\end{theorem}

Taking $\varepsilon =\dim_{\rm H} \mu - \dim_{\rm H}  \mu_\alpha$, it is easy to see that for all $\alpha \in \mathcal{A}_{d,k}$, $\Pi_\alpha$ is almost surely dimension conserving (as in \eqref{furst_dim_cons}) for the fractal percolation set $E$.

Note that the authors of \cite{falconer2014exact} claim dimension conservation almost surely for any fixed direction, but not almost surely for all of them simultaneously. In a following paper \cite{falconer2015dimcons} the same authors obtain dimension conservation for all directional projections at once in the sense of \eqref{weak_dim_cons}.

\begin{theorem}\cite[Theorem 3.4]{falconer2015dimcons}\label{thm:weak_dim_cons}
Almost surely, for all $\alpha \in \mathcal{A}_{d,k}$ $\Pi_\alpha$ is weakly dimension conserving for the fractal percolation set $E$.
\end{theorem}
Since the proof is beautiful, we sketch its idea below:
\begin{proof}[Sketch of the Proof of Theorem \ref{thm:weak_dim_cons}]
  Fix a
\begin{equation}\label{a83}
  p>\frac{1}{M^{d-k}}.
\end{equation}
and fix a small  $\varepsilon>0$ satisfying $p>M^{\varepsilon}/M^{d-k}$.
In order to study $E(p)$ we  will use $E(p'), E(\widetilde{p})$ where $p'$ and $\widetilde{p}$ are defined by
\begin{equation}\label{a82}
  \widetilde{p}:=\frac{M^\varepsilon}{M^{d-k}}=p \cdot p'.
\end{equation}

We write
$(\Omega,\mathcal{B},\mathbb{P}_{p})$
$(\Omega',\mathcal{B},\mathbb{P}_{p'})$
$(\widetilde{\Omega},\mathcal{B},\mathbb{P}_{\widetilde{p}})$ for the probability space
corresponding to the construction of $E(p), E(p'), E(\widetilde{p})$ respectively. It is easy to see that
\begin{equation}\label{a81}
  E(\widetilde{p})\stackrel{d}{=}E(p)\cap E(p').
\end{equation}
Hence, for two independent realizations $E_\omega(p)$ and $E_{\omega'}(p')$:
\begin{equation}\label{a80}
  \dim_{\rm H} \left(E_\omega(p)\cap E_{\omega'}(p')\right)=
  \frac{\log\left(M^d \cdot \widetilde{p}\right)}{\log M}=
  d+\frac{\log\widetilde{p}}{\log M}=k+\varepsilon.
\end{equation}
The following Lemma appeared as
\cite[Corollary of Lemma 2.1]{falconer2015dimcons}:
\begin{lemma}\label{a84}
Let $A  \subset [0,1]^d$ be  a Borel set and let  $\alpha\in \mathcal{A}_{d,k}$.
Assume that
\begin{equation}\label{a78}
  \mathbb{P}'\left\{\omega':\mathcal{L}_k
\Pi_\alpha\left(A\cap E_{\omega'}(p')\right)>0
\right\}>0.
\end{equation}
Then
\begin{equation}\label{a79}
  \mathcal{L}_k\left\{
  x\in S_\alpha:
  \dim_{\rm H} \left(A\cap \Pi_{\alpha}^{-1}(x)\right)>\frac{-\log p'}{\log M}
  \right\}>0.
\end{equation}
\end{lemma}
The authors of \cite{falconer2015dimcons} use this for a typical (conditioned on non-extinction)
realization $A:=E_\omega(p)$.
Namely,
by \eqref{a80}, the condition \eqref{a79} holds for all $\alpha\in\mathcal{A}_{d,k}$.
Then we obtain the assertion of Theorem \ref{thm:weak_dim_cons}
by \eqref{a79} and the observation:
$$
\frac{-\log p'}{\log M}
=
\frac{\log p}{\log M}+d-k-\varepsilon=\dim_{\rm H} (E(p))-k-\varepsilon.
$$
\end{proof}

The most recent and strongest results on dimension conservation are due to Shmerkin and Suomala \cite{shmerkin_si}, who investigated properties of Spatially Independent martingales (SI-martingales). SI-martingales include, e.g., multiplicative cascade measures, subdivision fractals and Poissonian cut-out measures, but here we focus only on the consequences of their results on fractal percolation, which is of the class of subdivision fractals. The following theorem states dimension conservation both for all orthogonal projections simultaneously and in the strong sense, under necessary conditions.

\begin{theorem}\cite[Corollary 12.10. (iii)]{shmerkin_si}\label{thm:strong_dim_cons}
Suppose that a fractal percolation set $E$ is constructed with homogeneous probabilities and that its expected dimension is greater than $k$ (in other words $p > M^{-(d-k)}$). Then almost surely, conditioned on $E \neq \emptyset$, for all $\alpha \in \mathcal{A}_{d,k}$ the projections $\Pi_\alpha$ are strongly dimension conserving for $E$.
\end{theorem}

In \cite{shmerkin_si} the authors generalise the key ideas of \cite{Peres2016} and complement it with many new techniques. We refer to \cite{shmerkin_si} for the proof of the above theorem as well as for an exhaustive study of SI-martingales which exceeds the scope of this work. Also note that many results in \cite{shmerkin_si} are also more general in the sense that those are not only valid for linear projections and intersections with lines (as in usual dimension conservation theorems), but, e.g., for other regular non-linear projections and intersections as well.

\subsection{Visibility} In this subsection we review the main result from
\cite{Arhosalo2012} about the visible parts of fractal percolation sets in $\mathbb{R}^2$. Let $\mathcal{A}$ be the set of straight lines on $\mathbb{R}^2$. Let $\Gamma$ be the natural Radon measure defined on $\mathcal{A}$ defined as follows:  For a $\theta\in[0,\pi)$, we write
$L _\theta$ for the straight line through the origin of angle $\theta$. Then for $A \subset \mathcal{A}$
\begin{equation}\label{a92}
  \Gamma(A):=\int\limits_{\theta\in[0,\pi)}
  \mathcal{L}_1\left(a\in L_{\theta^{\bot}}:a+L_\theta\in A\right)d\theta.
\end{equation}

\begin{definition} \label{a91} Let
$\ell \in\mathcal{A}$. For an $a\in \mathbb{R}^2$ we write
 $\Pi_\ell (a) $ for the orthogonal projection of $a$ to $\ell $. Moreover for $a,b\in\mathbb{R}^2$ we write $[a,b]$ for the straight line segment which connects $a$ to $b$.
Let $F \subset \mathbb{R}^2$ be a compact set and $\ell \in\mathcal{A}$ such that $\ell \cap F= \emptyset $ and $x\in \mathbb{R}^2\setminus F$. We define
\begin{equation}\label{a90}
  V_\ell (F):=\left\{a\in F: \left[a,\Pi_\ell (a)\right]\cap F=\left\{a\right\}\right\},\qquad
  V_x(E):=\left\{a\in F:[a,x]\cap F=\left\{a\right\}\right\}.
\end{equation}
\end{definition}
Marstrand projection theorem says:
\begin{equation}\label{a89}
  \dim_{\rm H}\Pi_{L_\theta}(F)=\min \left\{\dim_{\rm H} F,1\right\}\quad
  \text{ for $\mathcal{L}$-a.a. } \theta\in[0,\pi),
\end{equation}
where $\Pi_{L_\theta}$  is the projection to $L_\theta$ which is the straight line of direction $\theta$ through the origin (as we mentioned above).
For a compact set $F$ the authors of \cite{jarvenpaa2003visible} proved a similar result:
\begin{theorem}[J\"arvenp\"aa,J\"arvenp\"aa,MacManus,O'Neil]
 Let $F \subset \mathbb{R}^2$ be a compact set with $\dim_{\rm H} F \leq 1$.Then for $\Gamma$-almost all $\ell $
and for
$\mathcal{L}^2$-almost all $x$
satisfying
\begin{equation}\label{a87}
 \ell \cap F= \emptyset \text{ and } x\in\mathbf{R}^2\setminus F
\end{equation}
 we have
 \begin{equation}\label{a88}
   \dim_{\rm H} V_\ell (F)=\dim_{\rm H} F\text{ and }
   \dim_{\rm H} V_x=\dim_{\rm H} F.
 \end{equation}
On the other hand if $\dim_{\rm H} F > 1$ then for $\Gamma$-almost all $\ell $
and for
$\mathcal{L}^2$-almost all $x$ satisfying \eqref{a87}
we have
\begin{equation}\label{a86}
  1 \leq \dim_{\rm H} V_\ell(F) \text{ and } 1 \leq \dim_{\rm H} V_x(F).
\end{equation}
\end{theorem}
It is an open question if we can reverse the inequalities in \eqref{a86} in general. However, it was proved in \cite{Arhosalo2012} that this is possible in the case of the Fractal percolation sets.
\begin{theorem}[Arhosalo,J\"arvenp\"aa,J\"arvenp\"aa,Rams, Shmerkin]
  Assume that $p>\frac{1}{M}$ then for all straight line $\ell $ satisfying $\ell \cap E= \emptyset $ we have
  \begin{equation}\label{a85}
    \dim_{\rm H} V_\ell (E)=\dim_{\rm B}V_\ell (E)=1\quad \text{ holds for almost all realizations.}
  \end{equation}

\end{theorem}

%%%%%%%%%%%%%%%%%%%%%%%%
%%% New results %%%
%%%%%%%%%%%%%%%%%

 \section{New results}
As mentioned above,
Rams and Simon \cite{Rams} proved in Theorem \ref{thm:interval} that on the plane,  a fractal percolation having expected Hausdorff dimension greater than $1$, has an interval in its orthogonal projections of all of those directions in which condition $A(\alpha)$ holds (conditioned on non-extinction). Then the authors of \cite{Rams} proved that in the case of homogeneous probabilities $A(\alpha)$ holds for all directions but the directions of coordinate axes (which exceptional  projections had been treated by Grimmet and Falconer \cite{Falconer1992}). Now we point out that in the case of inhomogeneous  probabilities this does not always hold.

\begin{theorem}\label{newResults}
\
\begin{enumerate}
	\item[(i)] For any $N \in \mathbb{N}\setminus \{0\}$ there are $M$ and $\mathbf{p}$ such that a.s. $\dim_{\rm H}  E(M, \mathbf{p}) > 1$, if $E \neq \emptyset$, but
	$$
	\#\big\{ \alpha\in [0, \pi] \ :\ int(\Pi_\alpha E) = \emptyset\big\} \geq N.
	$$
	\item[(ii)] Moreover, there are $M$ and $\mathbf{p}$ (given explicitly) such that a.s. $E(M, \mathbf{p})$ has no exceptional direction in $\alpha\in(\pi/4,\pi/2)$ in the above sense if $E \neq \emptyset$, but $\alpha = 0$ is exceptional.
\end{enumerate}
\end{theorem}

From the above theorem we see that when a fractal percolation set is constructed with inhomogeneous probabilities not all of its projections are as regular as in Theorem \ref{thm:interval}. Nevertheless, we also propose a way of extending Theorem \ref{thm:interval} and make a corresponding conjecture which, if shown to be true, would imply the existence of intervals for all directions but a set of Hausdorff dimension $0$ for fractal percolation sets constructed with inhomogeneous probabilities. Below we prove Theorem \ref{newResults}.

\subsection{Randomized Sierpi\'nski carpet: exceptional directions for the existence of intervals}\label{noInt}
The most natural random Sierpi\'nski carpet is defined as follows:
\begin{example}\label{a72}
 Let $E_n$ be the $n$-th approximation of the fractal percolation set $E = E(d,M,\mathbf{p})$ constructed with $d=2$, $M=3$ and probabilities
\begin{equation*}
\begin{cases}
p_{1,1} = 0,\\
p_{i,j} = p > 3/8\ \text{ for $(i,j)\neq (1,1)$.}
\end{cases}
\end{equation*}
\end{example}

By \eqref{a95} we have $\dim_{\rm H} E > 1$ a.s. given $E\neq \emptyset$. In this subsection we investigate both the Sierpi\'nski carpet $\Lambda$, which is a deterministic set, and the random percolation set $E$ defined in Example \ref{a72}. In addition, let $\Lambda_n$ stand for the $n$-th approximation of $\Lambda$. The results of this section are based on the paper of Manning and the first author \cite{manning2013dimension}. For $x\in[-\cot(\alpha),1]$ let $\Lambda_{\alpha,x}$ stand for the slice of $\Lambda$ in direction $\alpha$ through $x$, that is
\[
  \Lambda_{\alpha,x}=\Lambda\cap \ell_\alpha(x),
\]
where, as we already mentioned, $\ell_\alpha(x)$ is the line making angle $\alpha$ with the $x$-axis, and going through $x$.

\begin{theorem}\label{0117}\cite[Theorem 9]{manning2013dimension}
For any fixed $\alpha \in [0,\pi/2]$ such that $\tan (\alpha)\in \mathbb{Q}$ and for Lebesgue almost all $x\in[-\cot(\alpha),1]$ we have
\[
\dim_{\rm B}    \Lambda_{\alpha,x}=d_{\alpha}<\frac{\log 8}{\log 3}-1,
\]
where $\dim_{\rm B} $ stands for the box dimension and $d_{\alpha}$ is a constant not depending on $x$.
\end{theorem}

\begin{proposition}\label{0114}
  For any $\alpha \in [0,\pi/2]$ such that $\tan \alpha\in \mathbb{Q}$ there exists a  $p_\alpha > 3/8$ such that
for $p\in\left(\frac{3}{8},p_\alpha\right)$
 the projected percolation set $\Pi_\alpha E$ has empty interior almost surely.
\end{proposition}

Clearly, Proposition \ref{0114} implies Theorem \ref{newResults} i.

\begin{proof}[Proof of Proposition \ref{0114}]
  We show that almost surely on a dense set of points $x$ we have $\ell_\alpha(x)\cap E=\emptyset$. By Theorem \ref{0117} for $\mathcal{L}_1$ almost all $x$ we have
	\begin{equation}\label{0120}
	\dim_{\rm B} \Lambda_{\alpha,x}=\frac{\log\big((8/3)(1-\varepsilon_\alpha)\big)}{\log 3}
	\end{equation}
	with an $\varepsilon_\alpha > 0$, so now on we suppose that $x$ has this property. Let
	\[
		\Lambda_{\alpha,x,n} = \Big\{ (\underline{i}_n,\underline{j}_n) \in \big\{\{0,1,2\}^2\setminus (1,1)\big\}^n\ :\ \ell_\alpha(x)\cap K_{\underline{i}_n,\underline{j}_n} \cap \Lambda \neq \emptyset \Big\},
	\]
	In addition, let $N_{\alpha,x,n} = \#\Lambda_{\alpha,x,n}$ and $\mathcal{E}_{\alpha,x,n} = \mathcal{E}_n \cap \Lambda_{\alpha,x,n}$.
	\begin{fact}
	$\mathbb{E}\big(\mathcal{E}_{\alpha,x,n}\big) \to 0$ as $n\to \infty$.
	\end{fact}
	\begin{proof}[Proof of the Fact]
	Using \eqref{0120}, there exists $n'=n'(\alpha,x)$ such that for $n>n'$ we have
	\[
	  \frac{\log N_{\alpha,x,n}}{\log 3^n} < \frac{\log\big((8/3)(1-\varepsilon_\alpha)(1+\varepsilon_\alpha)\big)}{\log 3},
	\]
	that is
	\[
	  N_{\alpha,x,n} < \left(\frac{8}{3}\right)^n\Big(1-\varepsilon_\alpha^2\Big)^n.
	\]
Using this and the definitions above, we obtain
	\[
	  \mathbb{E}\big(\#\mathcal{E}_{\alpha,x,n}\big) < N_{\alpha,x,n}p^n \to 0,
	\]
	if $p <p_\alpha:= \frac{3}{8}\frac{1}{1-\varepsilon_\alpha^2}$.
	\end{proof}
	
	To conclude the of the proposition we use Markov's inequality and write
	\[ \mathbb{P}\big(\#\mathcal{E}_{\alpha,x,n}>0\big)=\mathbb{P}\big(\#\mathcal{E}_{\alpha,x,n}\geq 1\big) \leq \mathbb{E}\big(\#\mathcal{E}_{\alpha,x,n}\big) \to 0.
	\]
	Hence
	\[
	\mathbb{P}\big(E\cap \ell _{\alpha}(x)\neq \emptyset\big) = 0.
	\]
	This holds for all $x$ satisfying \eqref{0120}, and also simultaneously for a dense countable set of such $x$.

\end{proof}

\begin{remark}
B{\'a}r{\'a}ny and Rams \cite[Theorem 1.2 and Proposition 2.1]{barany2014dimension}
 proved the analog of Theorem \ref{0117} for generalized Sierpi{\'n}ski-like carpets under a light non-resonance   assumption. Therefore the conclusion of Theorem \ref{0114} holds for all such carpets as well.
\end{remark}

\subsection{Conditions $A(\alpha)$ and $B(\alpha)$}

The part of the proof of the main result in \cite{Rams} which does not generalize to all inhomogeneous cases is the verification of a condition called Condition $B(\alpha)$. To define it, we introduce some notation. Recall that we fixed $d=2$. From this point on, we think of $\Pi_\alpha$ as the angle-$\alpha$ projection from $\mathbb{R}^2$ to the $x$-axis.
For $\mathfrak{i}\in \mathcal{I}^n$ let
\begin{equation*}
	\psi_{\alpha,\mathfrak{i}}:\mathbb{R}\supset \Pi_\alpha(K_\mathfrak{i}) \to \Pi_\alpha(Q)\subset\mathbb{R},\quad
\psi_{\alpha,\mathfrak{i}}:=\Pi_\alpha\circ S_\mathfrak{i}^{-1}\circ\Pi_\alpha^{-1},
\end{equation*}	
Note that this is well defined since $S_\mathfrak{i}$ (defined in Section \ref{a74}) does not contain rotation.
The support $\Pi_\alpha(K_{\mathfrak{i}})$ of $\psi_{\alpha,\mathfrak{i}}$ is illustrated on Figure \ref{fig1}\textbf{(a)}.

\begin{figure}
	\includegraphics[width=0.19\textwidth, valign=t]{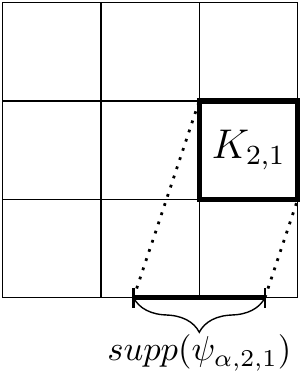}
	\hspace{1cm}
	\includegraphics[width=0.3\textwidth, valign=t]{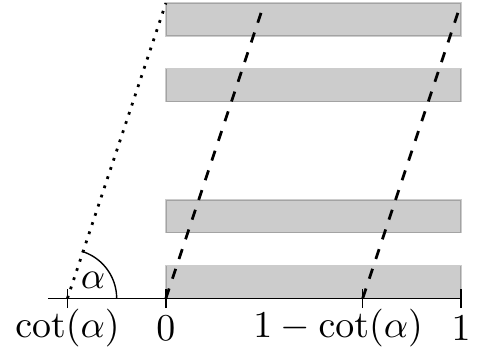}
	\hspace{1cm}
	\includegraphics[width=0.305\textwidth, valign=t]{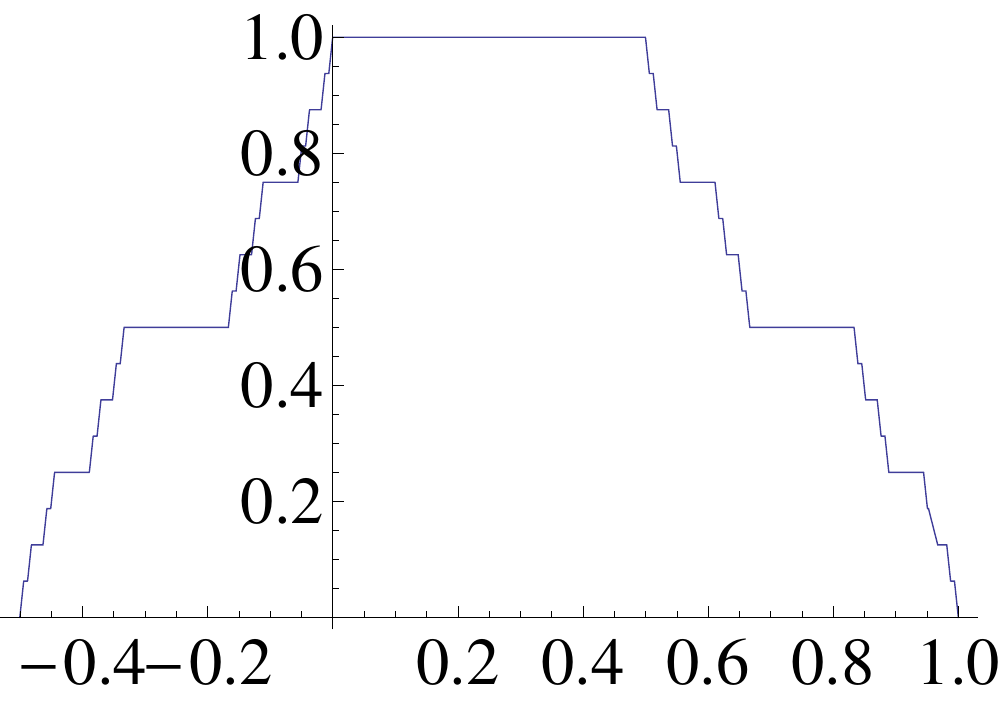}
	\caption{\textbf{(a)}: The support of $\psi_{\alpha,\mathfrak{i}}$. \textbf{(b)}: The approximation of the corresponding  deterministic measure  \textbf{(c)}: The function defined in \eqref{210} with $\cot(\alpha) = 0.5$.}\label{fig1}
\end{figure}

The following operator $F_\alpha$ \cite{Rams} is defined on non-negative real functions on $\Pi_\alpha(Q)$ (an interval), vanishing on the endpoints, by
\begin{equation}\label{0336}
  F_\alpha g(x)=\sum_{\mathbf{i} \in \mathcal{I}}{p_{\mathbf{i}}\cdot g\circ \psi_{\alpha,\mathbf{i}} (x)}.
\end{equation}
It is easy to see that the $n$-th iterate of $F_\alpha$ equals
\[
  F_\alpha^n g(x)=\sum_{\mathfrak{i}\in \mathcal{I}^n}{p_{\mathfrak{i}}\cdot g\circ \psi_{\alpha,\mathfrak{i}} (x)}.
\]
Note that by definition $\psi_{\alpha,\mathfrak{i}}$ is supported on $\Pi_\alpha(K_{\mathfrak{i}})$, therefore the sum is essentially taken over $\{\mathfrak{i}\in \mathcal{I}^n\ :\ x \in \Pi_\alpha(K_{\mathfrak{i}})\}$.

\begin{definition}[\cite{Rams}]\label{u99}
  We say that Condition $A(\alpha)$ holds if there exist $I_{1}^{\alpha},I_{2}^{\alpha} \subset \mathrm{int}(\Pi_\alpha(Q))$ closed intervals and a positive integer $r_\alpha$ such that
  \begin{description}
    \item[(a)] $I_{1}^{\alpha} \subset \mathrm{int}(I_{2}^{\alpha})$;
    \item[(b)] $F_{\alpha}^{r_\alpha}\ind_{I_{1}^{\alpha}} \geq \ind_{I_{2}^{\alpha}}$.
  \end{description}

\end{definition}
It is a simple observation that $A(\alpha)$ is an open condition (if $A(\alpha_0)$ holds then $A(\alpha)$ holds for all $\alpha$ sufficiently close to $\alpha_0$.) In some cases (for example when $E$ is homogeneous in probabilities) to check that Condition $A(\alpha)$ holds for an $\alpha$ it is more convenient to check the following stronger condition:

\begin{definition}[\cite{Rams}]\label{Balpha}
We say that Condition $B(\alpha)$ is satisfied if there exists a function $g_\alpha$ such that
\begin{description}
  \item[(a)] $g_\alpha$ is continuous,
  \item[(b)] $g_\alpha(x) > 0$ for $x \in {\rm int}(\Pi_\alpha(Q))$,
  \item[(c)] $g_\alpha(x) = 0$ for $x \notin {\rm int}(\Pi_\alpha(Q))$ and
  \item[(d)] there exists an $\varepsilon>0$ such that for every $x\in {\rm int}(\Pi_\alpha(Q))$ we have
\end{description}
\begin{equation*}
  F_\alpha g_\alpha(x) > (1+\varepsilon)g_\alpha(x)
\end{equation*}
for an $\varepsilon > 0$.
\end{definition}

\begin{lemma}[\cite{Rams}]
For all $\alpha$, Condition $B(\alpha)$ implies Condition $A(\alpha)$.
\end{lemma}

Importantly, if $\int_{x \in \Pi_\alpha(Q)}{ g_\alpha(x) dx} = 1$, then $\int_{x \in \Pi_\alpha(Q)}{ F_\alpha g_\alpha(x) dx} = \sum_{\mathbf{i} \in \mathcal{I}}{p_{\mathbf{i}}} / M$, which is greater than $1$ if and only if the expected dimension of $E$ is greater than $1$. A natural candidate function for $g_\alpha$ in condition $B(\alpha)$ is the eigenfunction $f_\alpha$ of $F_\alpha$ corresponding to the eigenvalue $\sum_{\mathbf{i} \in \mathcal{I}}{p_{\mathbf{i}}} / M$:
\begin{equation}\label{a99}
  F_\alpha f_\alpha(x)=  \sum_{\mathbf{i} \in \mathcal{I}}{p_{\mathbf{i}}\cdot f_\alpha\circ \psi_{\alpha,\mathbf{i}} (x)}=\frac{\sum_{\mathbf{i} \in \mathcal{I}}p_{\mathbf{i}}}{M} \cdot  f_\alpha (x)
\end{equation}
 if such a function $f_\alpha$ exists and if it satisfies conditions (a)-(c) of Definition \ref{Balpha}.
We can characterize this function  $f_\alpha$ also in the following way:
Consider the  homogeneous self-similar IFS  on $\Pi_\alpha(Q)$ with common contraction ratio  $1/M$:
\begin{equation}\label{a98}
  \Psi_\alpha:=\left\{ \psi_{\alpha,\mathbf{i}}^{-1} (x)\right\}_{\mathbf{i} \in \mathcal{I}} \mbox{ with weights } q_{\mathbf{i}} = p_{\mathbf{i}} / \sum_{\mathbf{i} \in \mathcal{I}}{p_{\mathbf{i}}}.
\end{equation}

It follows from the change of variables formula
and \eqref{a99} that if there exists an $L^1$  function $f_\alpha $  with $f_\alpha \geq 0$ and $\int_{\Pi_\alpha(Q)}f_\alpha(x)dx=1$ which satisfies
\eqref{a99} then  this function $f_\alpha$ is the density function
of the IFS $\Psi_\alpha$ with weights $\left\{q_{\mathbf{i}}\right\}_{\mathbf{i} \in \mathcal{I}}$. Consequently, $f_\alpha$ is the density function
of the measure
$\nu_\alpha = \Pi_\alpha * \nu$.
This implies that
\begin{lemma}\label{a97}
  If the deterministic self-similar measure corresponding to the IFS $\Psi_\alpha$ and weights $\left\{q_{\mathbf{i}}\right\}_{\mathbf{i}\in\mathcal{I}}$ (defined in \eqref{a99}) is absolute continuous with a continuous  density function which is positive at $x$ if and only if $x\in\mathrm{int}(\Pi_\alpha(Q))$ then
  Condition $B(\alpha)$ holds.
\end{lemma}

 In the following we consider an example which satisfies the assumptions of Lemma \ref{a97}
 with a density which is an affine transform a the Devil' staircase function.

\subsection{A Cantor-like random carpet}
Now we work on the plane and
we present an example for which  Condition $B(\alpha)$ can be satisfied for all $\alpha\in\left[\frac{\pi}{4},\frac{3\pi}{4}\right]$. Hence, for almost all realizations, there is interval in $\Pi_\alpha(E)$, simultaneously  for all $\alpha\in\left[\frac{\pi}{4},\frac{3\pi}{4}\right]$, but
on the other hand, there is no interval in $\Pi_0(E)$.
\begin{example}\label{a73}
 Consider the random Sierpi\'nski-like carpet $E$ defined by $d=2$, $M=3$ and
\begin{equation*}
\begin{cases}
p_{1,j} = 0\text{ for all $j$,}\\
p_{i,j} = p > 1/2\ \text{ for all $j$ and all $i \neq 1$.}
\end{cases}
\end{equation*}
\end{example}
Clearly, $E$ is not simply the product of a random Cantor set and the unit interval, as the retention of squares of the same row is independent. However, as we noted earlier, the eigenfunction $f_\alpha$ of $F_\alpha$ can be characterized as the probability density function of the corresponding projected deterministic self-similar measure, that is the projection of the natural measure on the product of the unit interval and triadic Cantor set.

To give a formula for $f_\alpha$ let us denote the usual Cantor function (or Devil' staircase function)
by
\[
  \mathcal{C}(x) = \mu([0,x)),
\]
where $\mu$ is the natural measure of the deterministic triadic Cantor set on $[0,1]$. It is well known that $\mathcal{C}(x)$ is H{\"o}lder-continuous and grows monotonically from $0$ to $1$. Let us identify $\Pi_\alpha(K)$ with $[-\cot(\alpha), 1]$ for $\alpha \in (\pi/4, \pi/2)$ (on the $x$-axis).

We show that for $\alpha \in (\pi/4, \pi/2)$ the eigenfunction $f_\alpha$ of $F_\alpha$ is supported on the whole interval $(-\cot(\alpha), 1)$ and is given by (up to a normalizing constant)

\begin{equation}\label{210}
  f_\alpha(x) =
	\begin{cases}
			\mathcal{C}\big((x + \cot(\alpha))/\cot(\alpha)\big),\text{ if } x \in [-\cot(\alpha), 0),\\
			1,\text{ if } x \in [0, 1-\cot(\alpha)),\\
			\mathcal{C}\big((- x + 1)/\cot(\alpha)\big),\text{ if } x \in [1-\cot(\alpha), 1],
	\end{cases}
\end{equation}
for an illustration see Figure \ref{fig1}\textbf{(c)}. As for a proof it, we refer to Figure \ref{fig1}\textbf{(b)}: since the slices between the two dashed lines on Figure \ref{fig1}\textbf{(b)} are identical, hence $f_\alpha(x)$ is constant on $[0, 1-\cot(\alpha)]$, and that constant is finite and positive because $f_\alpha$ is a probability density, see Figure \ref{fig1}\textbf{(c)}. On the other hand, to see what happens for $x\in[-\cot(\alpha), 0]$ let us approximate the deterministic $2$-dimensional measure $\nu$ the usual way, i.e. with the normalized Lebesgue measure on the level-$n$ cylinder set (gray stripes on Figure \ref{fig1}\textbf{(b)}). Then clearly the angle $\alpha$ projection of its density $f_{\alpha,n}(x)$ on $[-\cot(\alpha), 0]$ is the $n$-th approximation of the Cantor function rescaled to that interval. Thus, in the limit $n\to \infty$ we receive \eqref{210}. The interval $[1-\cot(\alpha), 1]$ can be handled by symmetry.

Since $f_\alpha(x)$ is positive and continuous on $(-\cot(\alpha), 1)$ we can apply Condition $B$ (\cite[Theorem 3 and Proposition 9]{Rams}) to obtain that almost surely, if $E\neq \emptyset$, $\Pi_\alpha E$ contains an interval for all $\alpha \in (\pi/4, \pi/2)$, and by symmetry a.s. for all $\alpha \in (\pi/4, \pi/2)\cup (\pi/2, 3\pi/4)$. Moreover, according to \cite[Theorem 1]{Falconer1992}, $\Pi_{\pi/2}(E)$ too contains an interval if $E\neq \emptyset$, almost surely.

On the other hand, $\Pi_{0}(E)$ is contained in the triadic Cantor set and in particular contains no interval.

\subsection{Checking Condition B}

It turns out that checking Condition B is much harder in general, and it is not even true for all directions, as we saw in the previous example and in Section \ref{noInt}. Consider the random Sierpi\'nski carpet of Section \ref{noInt}. It is known, that the projection of the corresponding deterministic self-similar measure is singular on a residual set of directions \cite{SimonVago}, so in particular the exceptional set has Packing dimension $1$. On the other hand, we also know that $f_\alpha$ is in $L^p$ for some $p>1$ outside of an exceptional set of Hausdorff-dimension $0$ \cite{shmerkin2014absolute}. Motivated by these, Marstrand's theorem (\cite{Marstrand1954}) and the above examples, we make the following conjecture:
\begin{conjecture}\label{0333}
Operator $F_\alpha$, defined in \eqref{0336}, has an eigenfunction $f_\alpha$ with eigenvalue $\sum_{\mathbf{i} \in \mathcal{I}}{p_{\mathbf{i}}} / M$, which is continuous for all $\alpha\in [0, \pi]$, except for a set of directions of $0$ Hausdorff dimension.
\end{conjecture}

Verification or disproof of this conjecture would not only provide an advance in the understanding of fractal percolation, but is also of independent interest.

Now we propose a possible (but by far incomplete) way of proving this conjecture. For simplicity assume we have a random Sierpi\'nski-like carpet, i.e., for all $\mathbf{i} \in \mathcal{I}$ we have either $p_{\mathbf{i}} = 0$ or $p_{\mathbf{i}} = p$ with some $p\in (0,1]$, and we write
$$
\Sigma_n = \mathcal{I}^n \text{ and } \Sigma = \mathcal{I}^\mathbb{N},
$$
and the words corresponding to non-zero probabilities are denoted by
$$
\Omega_n = (\mathcal{I_+})^n \text{ and } \Omega = (\mathcal{I_+})^\mathbb{N},
$$
where $\mathcal{I} = \{ 0, \dots, M-1\}^d$, $d=2$, and $\mathcal{I_+} = \{ \mathbf{i} \in \mathcal{I}\ : p_{\mathbf{i}} \neq 0\ \}$. For symmetry reasons let us consider only directions $\alpha\in [\pi/4, \pi/2]$ of the projection, and let $\beta = \cot(\alpha)$. In fact it is possible to show that the general case (when not all $p_{\mathbf{i}}\in\left\{0,p\right\}$) is equivalent, by changing to a sufficiently high iterate and throwing away some of the retained cubes).

 For $x\in [-\beta,1]$ the level-$n$ cylinders of the slice through $x$ are denoted by
\begin{equation}\label{0321}
	\Omega_{\alpha,n,x}=\big\{\mathfrak{i}\in \Sigma_n \mid \ell_\alpha(x)\cap K_\mathfrak{i}\neq \emptyset \big\},
\end{equation}
where, as above, $\ell_\alpha(x)$ is the line going through $x$ in direction $\alpha$.
% Let $\mu_\alpha$ be the natural self-similar measure defined by $I_\alpha$ and the uniform measure on the cylinders, and let $f_\alpha(x)$ be its density, if it exists.

To show the continuity in Conjecture \ref{0333}, one could approximate $f_\alpha$ to level $n$ by $f_{\alpha,n}$ in the following way:
\begin{itemize}
	\item Let $f_{\alpha, 0}(x)$ be continuous, $f_{\alpha,0}(-\beta)=f_{\alpha,0}(1)=0$ and $f_{\alpha,0}(x)>0$ for $x\in (-\beta,1)$;
	\item For $n\geq 1$ let
	\[
	f_{\alpha,n}(x)=\frac{M}{\sum_{\mathbf{i}\in \mathcal{I}}p_\mathbf{i}}F_\alpha f_{\alpha,n-1}(x).
	\]
\end{itemize}
Then to get the continuity of $f_\alpha$ it is enough to show that $\left\|f_{\alpha,n+1}(x)-f_{\alpha,n}(x)\right\|_\infty \to 0$ exponentially fast in $n$ (which implies H{\"o}lder continuity). With $h(x)=f_{\alpha,1}(x)-f_{\alpha,0}(x)$ we can write
\begin{equation} \label{s02}
  \big| f_{\alpha,n+1}(x)-f_{\alpha,n}(x)\big| = |F^n h(x)|=\frac{M^{n}}{\#\Sigma_n}\bigg| \sum_{\mathfrak{i}\in \Sigma_{\alpha,n,x}}h\big(\psi_{\mathfrak{i}}(x)\big)\bigg|
	=\frac{M^{n}N_{\alpha,n,x}}{\#\Sigma_n}\bigg| \frac{1}{N_{\alpha,n,x}}\sum_{\mathfrak{i}\in \Sigma_{\alpha,n,x}}h\big(\psi_{\mathfrak{i}}(x)\big)\bigg|,
\end{equation}
where $\Sigma_{\alpha,n,x}$ was defined in \eqref{0321} and $N_{\alpha,n,x}=\#\Sigma_{\alpha,n,x}$.

To estimate the sum on the right hand side of \eqref{s02} note that $\int_{-\beta}^1 h(x)dx=0$. One could look at
\[
\frac{1}{N_{\alpha,n,x}}\sum_{\mathfrak{i}\in \Sigma_{\alpha,n,x}}h\big(\psi_{\mathfrak{i}}(x)\big)
\]
as a partial ergodic sum. Let us split $\Sigma_{\alpha,n,x}$ to two parts, $\Sigma_{\alpha,n,x}^S$ and $\Sigma_{\alpha,n,x}^W$, according to whether $\ell_\alpha(x)$ hits the southern or the western border of corresponding squares $K_{\mathfrak{i}}$, or equivalently, whether $\psi_{\alpha,\mathfrak{i}}(x) \in [0,1]$ or $\in [-\beta, 0)$. Here we consider only summation over $\Sigma_{\alpha,n,x}^S$, as the terms in $\Sigma_{\alpha,n,x}^W$ can be handled similarly. Notice that $\psi_{\mathfrak{i}}(x)$ is the relative entrance point of $\ell_\alpha(x)$ to the southern border of $K_{\mathfrak{i}}$. Also note that in each $1/M^n$ wide row of $Q$ there is at most one $K_{\mathfrak{i}}$ such that $\mathfrak{i} \in \Sigma_{\alpha,n,x}^S$. Therefore, if $K_{\mathfrak{i}}$ is in the $k$-th row ($k=0, 1, \dots$, from the bottom), then
$$
  \psi_{\mathfrak{i}}(x) = x + k\beta\ (\text{mod }1).
$$

Using the notation $T_\beta$ for the usual rotation with angle $\beta$ on $[0,1]$ we can write
$$
\sum_{\mathfrak{i}\in \Sigma_{\alpha,n,x}^S}h\big(\psi_{\alpha,\mathfrak{i}}(x)\big) = \sum_{k \in K} h\big(T_\beta^k x\big),
$$
where the summation goes over $k$-s such that $K_{\alpha,\mathfrak{i}}$ is in the $k$-th row for an $\mathfrak{i} \in \Sigma_{\alpha,n,x}^S$. Although transformation $T_\beta$ has nice properties for most $\beta$ choices (e.g. ergodic if $\beta$ is irrational), estimating the above sum is challenging, since the set over which the summation goes is sparse because the slices $\ell_\alpha(x) \cap E$ has dimension $<1$.

\bibliographystyle{plain}

\bibliography{biblo_5}

\end{document}